\documentclass[a4paper]{article}
\usepackage{amsmath,amsfonts,amssymb}
\usepackage{graphicx, epsfig}

\newtheorem{theorem}{Theorem}[section]
\newtheorem{lemma}[theorem]{Lemma}

\newtheorem{defn}[theorem]{Definition}

\newtheorem{rem}[theorem]{Remark}

\newenvironment{proof}[1][Proof]{\textbf{#1.} }
{\hfill\rule{0.5em}{0.5em}\medskip}
\newenvironment{proof*}[1][Proof]{\textbf{#1.} }{}

\def\Kerekjarto{Ker\'ekj\'art\'o } 

\def\epsilon{\varepsilon}

\begin{document}


\title{
Hairiness of $\omega$-bounded surfaces\\ with non-zero Euler
characteristic
\\}

\author{Alexandre Gabard}
\maketitle

\newbox\abstract
\setbox\abstract\vtop{\hsize 8.5cm \noindent

\footnotesize \noindent\textsc{Abstract.} A little complement
concerning the dynamics of non-metric manifolds is provided,
by showing that any flow on an \hbox{$\omega$-bounded} surface
with non-zero Euler character has a fixed point.}

\centerline{\hbox{\copy\abstract}}

\section{Introduction}\label{sec1}

In a previous paper \cite{GabGa_2011} we
adventured timidly into the dynamics of flows on non-metric
manifolds.
Several missing links
foiled the
sharpness of the conclusions, leaving great zones of swampy
darkness. The
present
note is supposed to
generalise the $2$-dimensional {\it
$\omega$-bounded}\footnote{
A (topological) space is said to be {\it $\omega$-bounded} if
the closure of any countable subset is compact. In the case of
manifolds it is equivalent to ask that
any Lindel\"of subset has a compact closure. [For
Lindel\"of$+$locally second countable $\Rightarrow$ second
countable $\Rightarrow$ separable.]} hairy ball theorem, which
in \cite[Thm\,4.5]{GabGa_2011} was confined
to the simply-connected case. By a {\it hairy ball} one
understands
commonly
a topological space such that any
{\it flow} (i.e., continuous
${\Bbb R}$-action) has a fixed point. In
\cite[Section~4]{GabGa_2011}  we
speculated that the hairy ball theorem for $\omega$-bounded
surfaces with Euler character $\chi\neq 0$ might be plagued by
the
existence of {\it wild pipes}. The latter jargon
(to which we shall not attempt to
assign a precise
meaning)
refers
to the phenomenology discovered
in Nyikos~\cite[\S 6, p.\,669--670]{Nyikos84}
effecting that weird continua can be
transfinitely amalgamated to construct (``wild'') long pipes
laking a {\it canonical}\footnote{This means primarily that
$M_\lambda=\bigcup_{\alpha < \lambda} M_\alpha$ for any limit
ordinal $\lambda$; we shall not use this, but the reader may
wish to compare Nyikos~\cite[Def.\,4.3, p.\,656]{Nyikos84}.}
$\omega_1$-exhaustion, whose levels closures are compact
annuli. In contradistinction,
Theorem~\ref{omega_bded_hairy_ball_gnal:thm} below
indicates rather that the ``tameness'' of $2$-dimensional
topology (Schoenflies) conjointly with the
Poincar\'e-Bendixson theory (which applies in the {\it
dichotomic} pipes, where Jordan separation holds true)
%
seem to
unite into a stronger force
supplanting the ``wildness'' of pipes, at least
as far as the hairy ball paradigm is concerned.
Thus, Conjecture~4.14 in \cite{GabGa_2011} (saying that
$\chi(M)\neq 0$ is a sufficient condition in the
$\omega$-bounded context for the existence of an
equilibrium point) holds
true for $n=\dim M=2$,
whereas
the Poincar\'e-Bendixson method used below
does not adapt to dimensions $n\ge 3$,
leaving the conjecture
wide unsettled.
%
%
In fact, a third more hidden
force
decides for the vacillation toward tameness, namely {\it
Whitney's flows}, i.e. the creation of a motion compatible
with a given oriented one-dimensional foliation. A noteworthy
feature of this
result of Whitney is that---albeit
 specifically metrical
(as amply discussed in \cite{GabGa_2011})---it
proves oft useful in non-metric
investigations (cf.
optionally \cite[Sec.\,2.2]{GabGa_2011}, where
its relevance to the classification of foliations on the long
plane ${\Bbb L}^2$
is recalled).

\section{Homological finiteness of $\omega$-bounded\\
low-dimensional manifolds (after Nyikos)}

To state properly Theorem~\ref{omega_bded_hairy_ball_gnal:thm}
below, we need an {\it a priori}
finiteness  of the (sing\-ular) homology of
$\omega$-bounded surfaces, especially of its Euler
characteristic. This section provides an elementary argument,
yet it
can be noticed that
the bagpipe theorem of Nyikos (which we shall anyway use
later) also implies the desired finiteness (via
Lemma~\ref{charact_bagpipe_equal_char_bag}). Thus the
economical reader may prefer to skip completely this section,
and move forward to Section~\ref{hairy_ball_sec:omega-bded}.

The following
argument of Nyikos  gives simultaneously the $3$-dimensional
case, for it
depends on the issue that metric manifolds in those low
dimensions ($\le 3$) admit PL-structures (piecewise linear).
(Below singular homology is understood, and coefficients
may be chosen in the fields ${\Bbb Q}$ or ${\Bbb F}_2={\Bbb
Z}/2{\Bbb Z}$.)

\begin{lemma}
{\rm (i)} An $\omega$-bounded $n$-manifold, $M$,
of dimension $n\le 3$ has
finite-dimensional homology in each degree.  {\rm (ii)}
Besides, the groups $H_i(M)$
vanish for $i> n$ (and also for $i=n$ if $M$ is
open). Thus,
 $M$ has a well-defined Euler
characteristic, $\chi(M)=\sum_{i=0}^{\infty} (-1)^i b_i(M)$,
where the Betti numbers $b_i(M)$ are the dimensions of the
$H_{i}(M)$.
%
\end{lemma}

\begin{proof} (We recall briefly the argument
of Nyikos \cite[Cor.\,5.11, p.\,665]{Nyikos84}.) The
assumption of low-dimensionality ($n\le 3$) ensures that
metric $n$-manifolds are ``triangulable'', or
better have PL-structures (Rad\'o \cite{Rado_1925},
Moise for $n=3$)\footnote{Of course nothing similar holds in
dimension $4$ (Rohlin 1952
\cite{Guillou-Marin-Rohlin-Casson_1952}, Freedman 1982).}.
Regular neighbourhood theory can be employed to engulf any
compact subset in a  compact bordered (polyhedral)
$n$-submanifold. ({\it Bordered manifold} means
manifold-with-boundary.) Thus starting with any chart $M_0$,
take its closure $\overline{M_0}$ which is compact (by
$\omega$-boundedness); cover it by charts to get a metric
enlargement $N_0$ (neighbourhood), in which one engulfs the
compactum $\overline{M_0}$
into a bordered compact manifold $W_0$, and let $M_1:={\rm
int} W_0$ be its interior. By transfinite induction of this
clever routine (letting $M_{\lambda}:=\bigcup_{\alpha<\lambda}
M_{\alpha}$ for any limit ordinal), one
constructs an $\omega_1$-exhaustion of $M$ by
open metric manifolds $M=\bigcup_{\alpha < \omega_1}
M_\alpha$, where for each {\it non-limit} ordinal $\alpha$ the
closure $\overline{M_{\alpha}}=W_{\alpha-1}$ is a compact
bordered $n$-manifold. (Unfortunately we cannot claim this at
limit ordinals!)

(i) If $M$ has
infinite-dimensional homology, there is
a countably infinite sequence $\gamma_n$ of homology classes
linearly independent in $H_i(M)$. The union of the supports of
representing cycles $c_n\in \gamma_n$ is Lindel\"of, thus
contained in some $M_\alpha$, $\alpha<\omega_1$. Therefore the
$c_n$'s define homology classes in
$H_i(\overline{M_{\alpha+1}})$ which are still linearly
independent, violating the finite dimensionality of
$H_i(\overline{M_{\alpha+1}})$.

(ii) This is a standard fact for which we may refer to
Samelson~\cite[Lemma~B]{Samelson_1965-homology} (compare also
Milnor-Stasheff~\cite[p.\,270--275]{Milnor-Stasheff_1957-1974}).
\end{proof}

{\small
\begin{rem} (Very optional reading.) {\rm We are not aware of a
corruption of the lemma in case $\dim M \ge 4$. Perhaps
 Nyikos' argument can be given more
ampleness if
instead of engulfing by polyhedrons
one tries to engulf
compacta by topological compact bordered manifolds (perhaps
somewhat akin to M.\,H.\,A. Newman 1966, The engulfing theorem
for topological manifolds, {\it Ann. of Math. (2) 84}).
Such manifolds or more generally compact metric ANR's are
dominated by finite polyhedra, thus ``finitary''
homologically. As we shall not use this
presently, we prefer to skip this delicate question.}
\end{rem}

}

\section{Non simply-connected $\omega$-bounded hairy
ball}\label{hairy_ball_sec:omega-bded}

\begin{theorem}~\label{omega_bded_hairy_ball_gnal:thm}
Any flow on an $\omega$-bounded surface with non-zero Euler
characteristic $\chi(M)\neq 0$
has a stationary point.
\end{theorem}

\begin{proof} Let $f$ be a flow on such a surface $M$.
By Nyikos \cite[Thm 5.14, p.\,666]{Nyikos84} the surface
admits a bagpipe decomposition $M=B\cup \bigsqcup_{i=1}^n
P_i$, where the {\it bag} $B$ is a compact connected
bordered\footnote{By a {\it bordered manifold} we mean a
manifold-with-boundary.} surface with $n$ {\it contours}
(=circular boundary components) and the $P_i$ are long pipes.
In slight
departure from Nyikos
\cite[Def.\,5.2, p.\,662]{Nyikos84}
our pipes are supposed to
have a boundary circle
which seems in better accordance with the combinatorial
``cut-and-paste'' philosophy. It is easy to check that
$\chi(B)=\chi(M)$, and that this equality holds for any
bagpipe decomposition of $M$. (For this numerology cf.
Lemma~\ref{charact_bagpipe_equal_char_bag} below, and for the
(modified) definition of a long pipe compare eventually the
discussion in Section~\ref{pipes} below. Of course
Nyikos bagpipe theorem is
not jeopardized by this minor change of viewpoint.)

Let
propagate the bag $B$ under the dynamics $f\colon {\Bbb R}
\times M \to M$, to obtain ${\Bbb R} B := f({\Bbb R}\times B)$
which is Lindel\"of.
By $\omega$-boundedness the closure $\overline{{\Bbb R} B}$ is
compact, and flow-invariant (yet,
unlikely to be a respectable bag; a priori only a weird
compactum stemming from a complicated diffusion process). At
any rate, the residual surface $S:=M-\overline{{\Bbb R} B}$ is
invariant and contained in $M-B$. Clearly, the set $M-B$
(consisting of the residual open pipes) is {\it dichotomic},
i.e., divided by any Jordan curve
(cf. Lemma~\ref{dichotomy_of_a_pipe_interior} below), hence by
heredity $S$ is
likewise dichotomic (compare \cite[Lemma 5.3]{GabGa_2011}).
This will allow us to apply the Poincar\'e-Bendixson theory in
each pipe to surger out a new flow-invariant bag.

\begin{figure}[h]
    \epsfig{figure=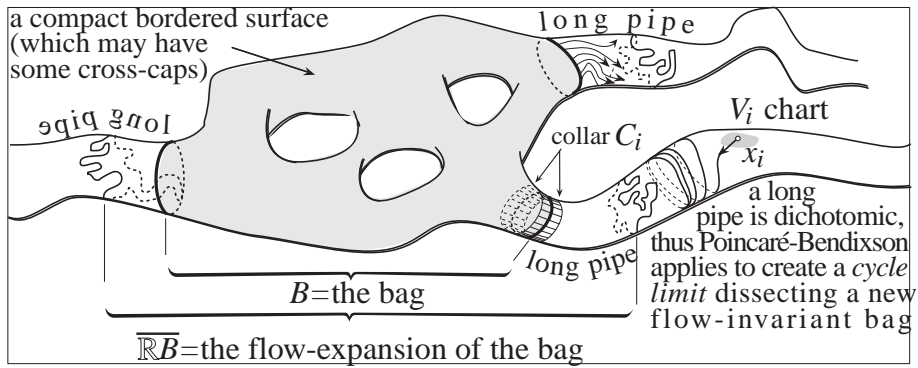,width=122mm}
  \caption{\label{Artist_views}
  A Nyikos' bagpipe
  with a Poincar\'e-Bendixson argument
  in the pipes}
\vskip-5pt\penalty0
\end{figure}

For each $i=1, \dots, n$, choose a ``remote'' point $x_i \in
P_i \cap S$ on the pipe $P_i$ and a chart $V_i\subset S$ about
it. (Such points $x_i$ exist, because the
compact set $\overline{{\Bbb R}B}$ cannot cover completely any
of the non-metric pipes $P_i$.) By $\omega$-boundedness the
orbit-closure
$\overline{{\Bbb R} x_i}^{M}$ (in $M$) is compact, yet a
priori not contained in ${\Bbb R}V_i$. To arrange the
situation we need a little trick, based mostly  on
Whitney~1933~\cite{Whitney33}.
We may assume that $f$ has no stationnary point, otherwise we
are finished. Take $C_i$ a little open collar of the circle
$\partial P_i$ such that $\overline{C_i}\approx{\Bbb
S}^1\times[0,1]$ and pairwise disjoint
($\overline{C_i}\cap\overline{C_j}=\varnothing$ if $i\neq j$).
Aggregate $C_i$ to ${\Bbb R} V_i$. This set ${\Bbb R} V_i \cup
C_i=:W_i$ is not flow-invariant a priori, but we may look at
the foliation $\cal F$ on $M$ induced by the {\it
non-singular} flow $f$ (Whitney$+$Hausdorff, as discussed in
\cite[Thm 2.2]{GabGa_2011}) and restrict ${\cal F}$ to the
open set
$W_i$, which is Lindel\"of, hence metric (Urysohn).
The theorem of (Ker\'ekj\'art\'o-)Whitney \cite[Thm
2.5]{GabGa_2011}\footnote{Compare
Ker\'ekj\'art\'o~1925~\cite{Kerekjarto_1925}, and
Whitney~1933~\cite{Whitney33}
as the
original sources.} creates a
flow-motion
$f_i$ compatible with the restricted foliation
${\cal F}\vert W_i$.

By Poincar\'e-Bendixson applied to $W_i$
acted upon by the flow
$f_i$ (cf. e.g., \cite[Lemmas~4.2~and~4.4]{GabGa_2011}),
$\overline{{\Bbb R}x_i}^M$ contains either a fixed point or a
periodic orbit, say
$K_i\approx {\Bbb S}^1$. The first option
cannot occur by construction.
The circle $K_i$ is certainly contained in the pipe $P_i$, yet
may touch its boundary, so we regard it in the slightly
enlarged pipe $P_i\cup \overline{C_i}=:P_i^{\star}$. Let
$B_{\star}=B-\bigcup_{i=1}^n {\rm int} P_i^{\star}$ be the
corresponding smaller ``retracted'' bag. Filling the pipe
$P_i^{\star}$ by a disc $D_i^{\star}$ yields a
simply-connected surface $F_i:=P_i^{\star}\cup D_i^{\star}$
(cf. Definition~\ref{pipe:def}).
By Schoenflies \cite[Prop.\,2]{GaGa2010} $K_i$ bounds a disc
$D_i\subset F_i$ containing $D_i^{\star}$ in its interior.
Thus $D_i-{\rm int}D_i^{\star}$ is an annulus $A_i\approx{\Bbb
S}^1\times[0,1]$ (again by Schoenflies, or more precisely by
its corollary known as the $2$-dimensional annulus theorem,
e.g., Moise~\cite[p.\,91]{Moise_1977}). Therefore the
non-metric component of $P_i^{\star}$ cut along
 $K_i$, that is $\Pi_i:=F_i-{\rm int}D_i$, is
 again a long pipe (fill it by the 2-disc
 $D_i$).
Thus, surgering
$M$ along the (disjoint) union
$\bigsqcup_{i=1}^n K_i$ yields a new
bagpipe decomposition, whose bag is $B^{\star}:= B_{\star}\cup
\bigsqcup_{i=1}^n A_i$ (annular expansion of the retracted
bag) and with pipes $\Pi_i$.
The new bag $B^{\star}$ is flow-invariant (under the original
flow $f$), thus a fixed point is
created by Lefschetz\footnote{Or
Poincar\'e to be historically sharper, yet confined to the
vector fields case.
}, since $\chi(B^{\star})=\chi(M)$ (by
Lemma~\ref{charact_bagpipe_equal_char_bag} below), which is
non-zero by assumption. (As usual one applies Lefschetz to the
dyadic times $t_n=1/2^n$ of the flow, to get a nested sequence
of non-void fixed-point sets $K_n={\rm Fix} (f_{t_n})$ (where
$f_t(x)=f(t,x)$), whose infinite intersection
$\bigcap_{n=0}^{\infty} K_n$ is non-empty by compactness of
the bag $B^{\star}$, and a point in this intersection is fixed
under all dyadic times, hence under all real times.)
\end{proof}

\section{Nyikos's bagpipe decompositions}\label{pipes}

To put the bagpipe philosophy of Nyikos in closer connection
to the classical combinatorial topology, it
seems convenient to alter slightly the original definition of
a pipe (given in Nyikos~\cite[Def.\,5.2, p.\,662]{Nyikos84}).
First, amending a boundary to the pipe gives some
material substrate for a sewing procedure
along the bag boundaries,
and second we may wish to express the ``pipe'' condition
intrinsically without
the artifact of
an
exhaustion (as already
implicit in Nyikos~\cite[p.\,668, \S\,6 and
p.\,644]{Nyikos84}).

\begin{defn}\label{pipe:def} {\rm A {\it long pipe} is a non-metric $\omega$-bounded
$2$-manifold $P$ with one boundary component $\partial
P\approx {\Bbb S}^1$ homeomorphic to the circle,
which capped-off by a $2$-disc $D$
yields a {\it simply-connected}  $P\cup D=:P_{\rm filled}$
(called the {\it filled pipe}).}
\end{defn}

\begin{lemma}\label{dichotomy_of_a_pipe_interior}
The interior of any long pipe is
dichotomic,
i.e., divided by any embedded circle (alias Jordan curve).
\end{lemma}

\begin{proof} The dichotomy of the
filled
pipe $P_{\rm filled}$ follows at once from the dichotomy of
any simply-connected surface
(cf. \cite[Prop.\,6]{GaGa2010}). Since the interior of the
pipe ${\rm int}P\subset P_{\rm filled}$, its dichotomy follows
by heredity
\cite[Lemma~5.3]{GabGa_2011}.
\end{proof}

\begin{lemma}\label{charact_pipe} The (singular) Euler
characteristic of any long pipe is $0$.
\end{lemma}

\begin{proof} By definition
the filled pipe $P\cup D=:P_{\rm filled}$ is {\it
simply-connected}, i.e. its fundamental group, hence a
fortiori its abelianisation $H_1$, is trivial. Thus,
$\chi(P_{\rm filled})=1-0+0=1$, for the second Betti number
$b_2(P_{\rm filled})=0$ via the classical vanishing of the
top-dimensional homology of an open manifold (cf.
Samelson~\cite[Lemma~D]{Samelson_1965-homology}). Via
Mayer-Vietoris one has $\chi(P\cup D)=\chi(P)+\chi(D)$ (like
in the combinatorial
setting). Thus, $\chi(P)=1-1=
0$.
\end{proof}

\begin{lemma}\label{charact_bagpipe_equal_char_bag} If a surface $M$ has a bagpipe decomposition
$M=B\cup \bigsqcup_{i=1}^n P_i$. Then $\chi (M)=\chi(B)$.
\end{lemma}

\begin{proof} Set $P=\bigsqcup_{i=1}^n P_i$, thus $M=B\cup P$.
The Mayer-Vietoris sequence shows that
$\chi(M)=\chi(B)+\chi(P)$.
By the obvious additivity of homology,
$\chi(P)=\sum_{i=1}^n\chi(P_i)$, where each individual pipe
has
$\chi(P_i)=0$ by
(\ref{charact_pipe}).
\end{proof}



{\small {\bf Acknowledgements.} The author
wishes to thank
David Gauld, Mathieu Baillif, Misha Gabard and Elias Boul\'e
for specialised
discussions, as well as general
conversations with  Claude Weber, Michel Kervaire, Jean-Claude
Hausmann, Andr\'e Haefliger and Felice Ronga all over the
years.

}

{\small

}

{
\hspace{+5mm} 
{\footnotesize
\begin{minipage}[b]{0.6\linewidth} Alexandre
Gabard

Universit\'e de Gen\`eve

Section de Math\'ematiques

2-4 rue du Li\`evre, CP 64

CH-1211 Gen\`eve 4

Switzerland

alexandregabard@hotmail.com
\end{minipage}
\hspace{-25mm} }


\begin{thebibliography}{30}








\bibitem{GaGa2010}
A. Gabard and D. Gauld, \textsl{Jordan and Schoenflies in
non-metrical analysis situs}, arXiv (2010).


\bibitem{GabGa_2011}
A. Gabard and D. Gauld, \textsl{Dynamics of non-metric
manifolds}, arXiv (2011).





\bibitem{Guillou-Marin-Rohlin-Casson_1952} L. Guillou, A.
Marin, \textsl{A la recherche de la topologie perdue, I Du
c\^ot\'e de chez Rohlin, II Le c\^ot\'e de Casson}, Progress
in Mathematics 62, Birkh\"auser, 1986.








\bibitem{Kerekjarto_1925}
B. de \Kerekjarto\!\!, {\em On a geometric theory of
continuous groups. I. Families of path-curves of continuous
one-parameter groups of the plane}, Ann. of Math. (2) 27
(1925), 105--117.



\bibitem{Milnor-Stasheff_1957-1974}
J.\,W. Milnor and J.\,D. Stasheff, \textsl{Characteristic
classes}, Annals of Mathematics Studies 76, Princeton
University Press,  1974. (Based on lectures in 1957.)


\bibitem{Moise_1977}
E.\,E. Moise,  \textsl{Geometric Topology in Dimensions 2 and
3}, Graduate Texts in Mathematics~47, Springer-Verlag, New
York, Heidelberg, Berlin, 1977.









\bibitem{Nyikos84}
P.\,J.~Nyikos, \textsl{The theory of nonmetrisable manifolds},
In: {\em Handbook of Set-Theoretic Topology}, ed.
 Kunen and Vaughan, 633--684, North-Holland,
Amsterdam, 1984.





\bibitem{Rado_1925}
T.~Rad\'o, \textsl{\"Uber den Begriff der Riemannschen
Fl\"ache}, Acta Szeged {2} (1925), 101--121.





\bibitem{Samelson_1965-homology} H. Samelson, \textsl{On Poincar\'e
duality}, J. Anal. Math.  14 (1965), 323--336.



\bibitem{Whitney33}
H. Whitney, \textsl{Regular families of curves}, Ann. of Math.
(2)  34 (1933), 244--270.


\end{thebibliography}
\end{document}